\documentclass[11pt,reqno]{amsart}
\usepackage{amssymb,mathrsfs,graphicx}
\usepackage{ifthen}
\usepackage{colortbl}
\definecolor{black}{rgb}{0.0, 0.0, 0.0}
\definecolor{red}{rgb}{1.0, 0.5, 0.5}

\provideboolean{shownotes} 
\setboolean{shownotes}{true} 
\newcommand{\margnote}[1]{
\ifthenelse{\boolean{shownotes}}%
{\marginpar{\raggedright\tiny\texttt{#1}}}%
{}%
}
\newcommand{\hole}[1]{
\ifthenelse{\boolean{shownotes}}%
{\begin{center} \fbox{ \rule {.25cm}{0cm} \rule[-.1cm]{0cm}{.4cm}
\parbox{.85\textwidth}{\begin{center} \texttt{#1}\end{center}} \rule
{.25cm}{0cm}}\end{center}} {} }

\title[]{Critical thresholds in 1D Euler equations with nonlocal forces}

\author[Carrillo]{Jos\'{e} A. Carrillo}
\address[Jos\'{e} A. Carrillo]{\newline Department of Mathematics
    \newline Imperial College London, London SW7 2AZ, United Kingdom}
\email{carrillo@imperial.ac.uk}

\author[Choi]{Young-Pil Choi}
\address[Young-Pil Choi]{\newline Department of Mathematics
    \newline Imperial College London, London SW7 2AZ, United Kingdom}
\email{young-pil.choi@imperial.ac.uk}

\author[Tadmor]{Eitan Tadmor}
\address[Eitan Tadmor]{\newline Center for Scientific Computation and Mathematical  Modeling (CSCAMM), Department of Mathematics, and Institute for Physical Science \& Technology, \ University of Maryland, College Park, USA}
\email{tadmor@cscamm.umd.edu}

\author[Tan]{Changhui Tan}
\address[Changhui Tan]{\newline Center for Scientific Computation and Mathematical  Modeling (CSCAMM) and Department of Mathematics, \ University of Maryland, College Park, USA}
\email{ctan@cscamm.umd.edu}



\numberwithin{equation}{section}

\newtheorem{theorem}{Theorem}[section]
\newtheorem{lemma}{Lemma}[section]
\newtheorem{corollary}{Corollary}[section]
\newtheorem{proposition}{Proposition}[section]
\newtheorem{remark}{Remark}[section]

\newcommand{\R}{\mathbb R}

\newcommand{\mc}{\mathcal C}

\newcommand{\bq}{\begin{equation}}
\newcommand{\eq}{\end{equation}}

\newcommand{\lt}{\left}
\newcommand{\rt}{\right}

\newcommand{\pa}{\partial}
\newcommand{\mw}{W}
\newcommand{\kk}{\pa_x^2K}
\def\charf {\mbox{{\text 1}\kern-.30em {\text l}}}

\def\u{u}
\def\x{x}
\def\y{y}
\def\grad{\pa_x}
\def\div{\pa_x}

\begin{document}
\allowdisplaybreaks

\date{\today}


\begin{abstract} We study the critical thresholds for the compressible
  pressureless Euler equations with pairwise attractive or repulsive
  interaction forces and non-local alignment forces in velocity in one
  dimension. We provide a complete description for the critical
  threshold to the system without interaction forces leading to a
  sharp dichotomy condition between global in time existence or
  finite-time blow-up of strong solutions. When the interaction forces
  are considered, we also give a classification of the critical
  thresholds according to the different type of interaction forces. We
  also remark on global in time existence when the repulsion is modeled by the isothermal pressure law.
\end{abstract}

\maketitle \centerline{\date}

\tableofcontents

%
%
%
%
\section{Introduction and statement of main results}\label{sec1}
We are concerned with the following 1D system of presureless Euler equations with nonlocal interaction and alignment forces\begin{subequations}\label{main-eq}
\begin{eqnarray}
&&\pa_t \rho + \pa_x (\rho u) = 0,\quad x \in \R, \quad t > 0, \label{main-eq:a} \\
&&\,\,\,\, \pa_t u + u\pa_x u =  \int_{\R} \psi(x-y)(u(y,t) - u(x,t))\rho(y,t)dy
- \pa_x K \star \rho,\label{main-eq:b}
\end{eqnarray}
\end{subequations}
subject to initial density and velocity
\[(\rho(\cdot,t),u(\cdot,t))|_{t=0} = (\rho_0, u_0).\]
Since the total mass is conserved in time, we may assume, without loss
of generality, that $\rho$ is a probability density function, i.e., $\|\rho(\cdot,t)\|_{L^1} = 1$.

The system involves two types of non-local forces arising in many different fields such as collective behavior patterns in mathematical biology, opinion dynamics, granular media and others. In the particular context of multi-agents interactions, the system \eqref{main-eq} arises as macroscopic descriptions for individual based models (IBMs) of the form
\[
\dot{x}_i=v_i,\quad\dot{v}_i=\frac{1}{N}\sum_{j=1}^N\psi(x_i-x_j)(v_j-v_i) -
\frac{1}{N}\sum_{j=1}^N\pa_{x_i}K(x_i-x_j),
\]
where the force consists of attractive-alignment-repulsive  interactions, under a ``three-zone''  framework proposed in \cite{Ao,Reyn,HW}. Starting from the basic system of IBMs, one can derive a kinetic description by BBGKY hierarchies or mean field limits, see \cite{HT,CCR,review} and the references therein. The hydrodynamic equations of the form \eqref{main-eq} are obtained by taking moments on the kinetic equations and assuming a closure based on a monokinetic distribution, see \cite{CDMBC,HT,CDP,review} for details. These hydrodynamic equations lead to numerical solutions which share common features with the original IBMs systems such as flocks and mills patterns as demonstrated in \cite{CKMT}.

The first term on the right of (\ref{main-eq:b}) represents a non-local \emph{alignment}, where $\psi\in \mw^{1,\infty}(\R)$ is the \emph{influence function} which is assumed symmetric and uniformly bounded 
\[
0 \leq \psi_m\leq \psi(x) = \psi(-x) \leq \psi_M.
\]
The second term on the right of (\ref{main-eq:b}) represents \emph{attractive and/or repulsive} forces, through a symmetric smooth enough \emph{interaction potential}. We will start by assuming the regularity $K \in \dot{\mw}^{2,1}(\R)$. If the potential is convex (resp. concave) in $x$, the forces are attractive (resp. repulsive). 

We begin our discussion with the case in which the particles are only driven by alignment. Setting the attraction/repulsion force $K\equiv0$, we arrive at a system of 1D mass and momentum Euler equations coupled with the alignment force, 
\begin{align}\label{eq:EA}
\begin{aligned}
&\pa_t \rho + \pa_x (\rho u) = 0,\cr
&\pa_t u + u\pa_x u = \int_{\R} \psi(x-y)(u(y,t) - u(x,t))\rho(y,t)dy. 
\end{aligned}
\end{align}
We refer \eqref{eq:EA} as the \emph{Euler-Alignment} system for
short. It is realized as the  hydrodynamic system of the Cucker-Smale
flocking model \cite{CS,HT}. 

The authors in \cite{TT}  have recently
shown that global regularity of Euler-Alignment system is determined
by the initial configurations.
They show that there are only two possible scenarios, depending on upper- and lower-thresholds $\sigma_+ > \sigma_-$. If the initial
data lie above the upper threshold in the sense that $\pa_xu_0(x)>\sigma_+$ for all $x \in \R$, then such initial
data lead to global smooth solutions which must flock; 
on the other hand, if the initial data lie below the lower threshold $\sigma_-$, namely, if there exists $x \in \R$ such that $\pa_xu_0(x)<\sigma_-$, then such supercritical initial data lead to finite time blowup of
solutions. 
Our first result, investigated in Section \ref{sec:EA}, refines this
\emph{critical threshold phenomenon} and quantifies the precise
threshold in this case of Euler-Alignment dynamics. 
\begin{theorem}\label{thm11}
Consider Euler-alignment system \eqref{eq:EA}. 
\begin{itemize}
\item {[Subcritical region]}. If $\pa_xu_0(x)\geq-\psi\star\rho_0(x)$ for
  all $x\in\R$, then the system has a global classical solution, 
$\displaystyle (\rho,u)\in
 \mc(\R^+; L^\infty(\R))\times \mc(\R^+;\dot{\mw}^{1,\infty}(\R))$.
\item {[Supercritical region]}. If there exists an $x$ such that $\pa_xu_0(x)<-\psi\star\rho_0(x)$, then the solution blows up in a finite time.
\end{itemize}
\end{theorem}

Theorem \ref{thm11}  is sharp in the sense that it provides a precise initial threshold characterized by the pointvalues of $\rho_0$ and $\pa_x u_0$, for  the following dichotomy of initial configurations: either subcritical initial data which evolve into global strong solutions, or supercritical initial data which will blow up in a finite time. Moreover, in Appendix \ref{sec:App-1} we show that for the more general setup of subcritical case, the solutions remain as smooth as the initial data permit. 

\begin{remark}
A subcritical condition which gives rise to global smooth solutions derived in \cite[Theorem 2.4]{TT}, requires 
$\sigma_- < \inf_x\pa_xu_0(x) < \sigma_+$,  where the thresholds  $\sigma_\pm\equiv \sigma_\pm(V_0)$ shown in \cite[Figure 1]{TT}, are now functions  of the initial variance of velocity $V_0$ but otherwise are uniform in $x$.  In particular, the function $\sigma_+$ is increasing and one checks that
$-\psi\star\rho_0\leq-\psi_m=\sigma_+(0)\leq\sigma_+(V_0)$. Similarly,
$\sigma_-$ is decreasing and  $-\psi\star\rho_0\geq\sigma_-(0) \geq \sigma_-(V_0)$.
Thus, Theorem \ref{thm11} extends both, the sub- and supercritical regions derived in \cite{TT}.
\end{remark}

\medskip
Next, we deal with the case of $K\not\equiv0$, i.e., there is an additional force through the attractive-repulsive potential $K$. We begin by analyzing the behavior of the solutions for the particular potential 
$K(x)=\frac{k}{2}|x|$, $k \in \R$. This special potential is the 1D \emph{Newtonian potential}, where
$\pa_x^2K=k\delta_0$ is the Dirac delta function. When there is no alignment force, $\psi  \equiv 0$,
the system coincides with the 1D pressureless Euler-Poisson equations
\begin{align}\label{1d_ep}
\begin{aligned}
&\pa_t \rho + \pa_x (\rho u) = 0,\cr
&\pa_t u + u \pa_x u = -k\pa_x\phi , \quad \pa_x^2\phi=\rho.
\end{aligned}
\end{align}
Critical thresholds of the system \eqref{1d_ep} were studied in \cite{ELT}, followed by a series of extensions on multi-dimensional systems \cite{LL,LT2,LTW}. The result in \cite{ELT} shows that the system has finite time blow-up in the attractive case, $k > 0$, while in the repulsive force, $k < 0$, there exists a critical threshold.

With the alignment force ($\psi \not\equiv0$), we can naturally expect that the solution tends to be smoother than for Euler-Poisson. In Section \ref{sec:EAP}, we investigate the critical threshold phenomenon for the following Euler-Poisson-Alignment system:
\begin{align}\label{1d-main-newton}
\begin{aligned}
&\pa_t \rho + \pa_x (\rho u) = 0,\cr
&\pa_t u + u \pa_x u = -k\pa_x\phi + \int_{\R} \psi(x-y)(u(y,t) - u(x,t))
\rho(y,t) dy, \quad \pa_x^2\phi =\rho.
\end{aligned}
\end{align}
\begin{theorem}\label{thm:EAP} Consider Euler-Poisson-Alignment system
  \eqref{1d-main-newton}. 
\begin{enumerate}
\item Attractive Poisson forcing $k>0$. An unconditional finite-time blow up  of the solution for all initial configurations.
\item Repulsive Poisson forcing $k<0$. We distinguish between two cases:
\begin{itemize}
\item {[Subcritical region]}. If $\pa_xu_0(x) > -\psi\star\rho_0(x)+\sigma_+(x)$ for all
  $x\in\R$, then the system has a global classical solution. Here,
 $\sigma_+(x)=0$ whenever $\rho_0(x)=0$ and elsewhere  $\sigma_+(x)$ is   the (unique) negative root of the equation
\[
\qquad \frac{1}{\rho_0(x)} -\frac{1}{\psi_M^2}\lt(k + \psi_M \sigma_+(x)/\rho_0(x) - ke^{\psi_M \sigma_+(x)/k\rho_0(x)}\rt) = 0, \ \rho_0(x)>0.
\]

\item {[Supercritical region]}. If there exists an $x$ such that 
\[\pa_x u_0(x) < - \psi \star \rho_0(x)+\sigma_-(x),\quad \sigma_-(x):= -\sqrt{-2k\rho_0(x)},\]
then the solution blows up in a finite time.
\end{itemize}
\end{enumerate}
\end{theorem}
\begin{remark}

1. In the attractive case, the blowup is ``unconditional'', independent of the
choice of initial configuration. It indicates that Poisson force
dominates the alignment force.

2. In the repulsive case, alignment force enhances regularity. Indeed, we have a larger
subcritical region than the case of $K \equiv 0$, as $\sigma_+(x)<0$.

3. If $\psi$ has a positive lower bound $\psi_m>0$, we can obtain a
better supercritical region for the repulsive case. In particular, the
threshold condition is sharp when $\psi$ is a constant. Consult Remark
\ref{rem:EAPsharp} below for details.
\end{remark}

\medskip
In Section \ref{sec:general}, we consider general potentials with
enough smoothness, $K\in\dot{\mw}^{2,\infty}(\R)$.
We show the following threshold conditions 
\begin{theorem}\label{thm13} Consider the system \eqref{main-eq} with 
alignment and attractive-repulsive forces. 
\begin{enumerate}
\item Attractive case, $\kk < 0$.  If there exists an $x$
  such that $\pa_x u_0(x) < -\psi \star \rho_0(x)$,
  then the solution blows up in a finite time.
\item Repulsive case $\kk > 0$. We distinguish between two cases. 
\begin{itemize}
\item {[Subcritical region]}.
If $\pa_x u_0(x) \geq -\psi \star \rho_0(x)$ for all $x\in\R$,
then the system has a global classical solution.
\item {[Supercritical region]}. 
if there exists an $x$ such that
\[
\pa_x u_0(x) < -\psi \star \rho_0(x)  -\sqrt{\|\kk\|_{L^\infty}},
\] 
then the solution blows up in finite time.
\end{itemize}
\end{enumerate}
\end{theorem}
\noindent
We note in passing that if $\psi$ admits a positive lower bound, $\psi_m > 0$, then we there is a refined critical threshold outlined in remark \ref{rem:genlower} below, which is irrespective of the sign of interaction force.

Finally, we close our discussion by considering the the alignment system with additional pressure term $p(\rho) := A\rho^\gamma$
\begin{align}\label{press-eq-1d}
\begin{aligned}
&\pa_t \rho + \pa_x(\rho u) = 0, \quad x \in \R, \quad t >0,\cr
&\pa_t u + u\pa_x u + \frac{\pa_x p(\rho)}{\rho} = -k \pa_x\phi+ \!\int_{\R}\!\psi(x-y)(u(y,t) - u(x,t))\rho(y,t)\,dy, \  \  \pa_x^2\phi=\rho.
\end{aligned}
\end{align}
 The presence of the pressure destroys the original characteristic structure and it is well known that the solutions of the isentropic Euler equation \eqref{press-eq-1d} with $\psi\equiv k=0$
develop singularities in a finite time, consult e.g.,  \cite{Chen} and the references therein. Thus, the regularizing effects of Poisson and alignment forcing  compete with the generic formation of singularities due to the pressure term.
In \cite{TW} it was shown that the critical threshold for global regularity of the perssureless Euler-Poisson equations \eqref{1d_ep} survives when pressure is being added. The question arises whether the Euler-Alignment \emph{with} pressure admits a critical threshold for global regularity. Specifically, in Section
\ref{sec:pressure}, we consider  the Euler-Alignment with isothermal pressure $p(\rho)=A\rho$,
\begin{align}
\begin{aligned}
&\pa_t \rho + \pa_x(\rho u) = 0, \quad x \in \R, \quad t >0,\nonumber \cr
&\pa_t u + u\pa_x u + A \pa_x\ln(\rho) = \!\int_{\R}\!\psi(x-y)(u(y,t) - u(x,t))\rho(y,t)\,dy.\nonumber
\end{aligned}
\end{align}
This Euler-Alignment system with isothermal pressure was rigorously
derived in \cite{KMT} as the
hydrodynamic limit of the kinetic Cucker-Smale system (under the
assumptions that $\pa_xu$ and $\pa_x\ln\rho$ are bounded). 
In Theorem \ref{thm:isothermal} below we prove global regularity for subcritical initial configurations in the special case of a constant influence function $\psi\equiv Const$. We also point out the difficulties to extend this result to general influence functions, which is left for a future investigation. The necessary {\it a priori} estimate on the isothermal Euler-Alignment system  is found in Appendix \ref{sec:App-2}. 

\section{Euler-Alignment system}\label{sec:EA}
In this section, we study the one dimensional Euler-Alignment system
\eqref{eq:EA}, without taking into account the interaction potential $K$.

Differentiate the momentum equation in $\eqref{eq:EA}$ with respect
to $x$, then   $v := \pa_x u$ satisfies\footnote{We suppress the time dependence whenever it is clear from the context.}  
\begin{align*}
\begin{aligned}
&\pa_t \rho + u\pa_x \rho = - \rho v,\cr
&\pa_t v + u\pa_x v + v^2 = - u \int_\R \pa_x \psi(x-y) \rho(y)dy - \int_\R \psi(x-y) \pa_t \rho(y) dy - v \int_\R \psi(x-y) \rho(y) dy,
\end{aligned}
\end{align*}
where we used 
\[
\int_\R \pa_x \psi(x-y) (u(y) - u(x))\rho(y)dy = - u(x) \int_\R \pa_x \psi(x-y) \rho(y)dy - \int_\R \psi(x-y) \pa_t \rho(y) dy.
\]
Here the symmetry assumption of the influence function $\psi$ is essential.
Consider the characteristic flow $x(a,t)$ associated to the velocity field $u$ defined by
\[
\frac{d}{dt} x(a,t) = u(x(a,t),t), \quad \mbox{with} \quad x(a,0) = a.
\]
Then along this characteristic flow we find
\begin{align*}
\begin{aligned}
&\pa_t\rho(x(a,t),t) = - \rho(x(a,t),t) v(x(a,t),t),\cr
&\pa_t (v(x(a,t),t) + (\psi \star \rho)(x(a,t),t)) = - v^2(x(a,t),t) - v(x(a,t),t)(\psi \star \rho)(x(a,t),t).
\end{aligned}
\end{align*}
Set $d = v + \psi \star \rho$. Then we again rewrite the above system:
\begin{subequations}\label{eq-prop1}
\begin{eqnarray}
&& \rho^\prime = - \rho v = -\rho(d- \psi \star \rho),\label{eq-prop1a}\\
&& d^\prime = - v(v + \psi \star \rho) = - d(d - \psi \star \rho),\label{eq-prop1b}
\end{eqnarray}
\end{subequations}
where $^\prime$ denotes the time derivative along the characteristic flow $x(a,t)$.
\begin{proposition}\label{prop1} Consider the equation
  \eqref{eq-prop1}. Then we have
\begin{itemize}
\item If $d_0 < 0$, then $d \to - \infty$ in finite time.
\item If $d_0 = 0$, then $d(t) = 0$ for all $t \geq 0$.
\item If $d_0 > 0$, then $d(t)$ remains bounded for all time, and $d(t)
  \to \psi \star \rho(t)$ as $t \to \infty$.
\end{itemize}
\end{proposition}

\begin{proof}
First case $d_0<0$. It is clear to have that  if $d_0 < 0$, then $d(t) \leq 0$ for all $t \geq 0$. Then it follows from $\eqref{eq-prop1b}$ that
$d^\prime \leq -d^2$ which in turn yields 
$\displaystyle d(t) \leq \frac{d_0}{t + d_0}$.
Hence, $d(t)$ blows up at $t^* \leq -d_0$.\newline
The second case $d_0=0$ is trivial.\newline
Finally the third case $d_0>0$. Note that if $d(t) \in (0, \psi \star \rho(t))$,
then $d^\prime(t) > 0$ thus $d(t)$ is increasing up to $\psi \star
\rho(t)$. On the other hand, if $d(t) > \psi \star \rho(t)$, then
$d(t)$ is decreasing up to $\psi \star \rho(t)$. 
Note that
\[
\|\psi \star \rho\|_{L^\infty} \leq \|\psi\|_{L^\infty}\|\rho\|_{L^1} = \psi_M < \infty\,.
\]
Thus $\psi\star\rho$ is bounded, and we  conclude with the third statement of the proposition.
\end{proof}

We can also trace the dynamics of $\rho$ along the characteristic flow
from \eqref{eq-prop1a}.
\begin{itemize}
\item If $\rho_0=0$, clearly $\rho(t)=0$ for all $t\geq0$.
\item If $\rho_0>0$, we set
$\beta := \frac{d}{\rho}$ and its dynamics along particle path is easily found to be
\[
\beta^\prime = \frac{d^\prime \rho - d \rho^\prime}{\rho^2} = \frac{1}{\rho^2}\big(-d(d - \psi \star \rho)\rho + d\rho(d-\psi\star\rho)\big) = 0.
\]
Thus  $\beta(t) = \beta_0$ for all $t \geq 0$ and $\rho$ remains proportional to $d$ along each path, $\rho(t) \beta_0 = d(t)$, 
\end{itemize}
As a conclusion we have the following complete description of the
critical threshold for system \eqref{eq:EA}, and Theorem \ref{thm11}
then follows as a direct consequence.
\begin{corollary}\label{cor11} Consider the Euler-Alignment system
  \eqref{eq:EA}. Then we have
\begin{itemize}
\item If $\pa_x u_0(a) < -\psi \star \rho_0(a) $, then $\pa_x
  u(x(a,t),t) \to -\infty$. Moreover, if $\rho_0(a)>0$,
  $\rho(x(a,t),t) \to +\infty$ in a finite time.
\item If $\pa_x u(a) = -\psi \star \rho(a)$, then $\pa_x u(x(a,t),t) = -\psi \star \rho(x(a,t),t)$ for all time $t \geq 0$.
\item If $\pa_x u_0(a) > -\psi \star \rho_0(a)$, then
  $\pa_xu(x(a,t),t)$ and $\rho(x(a,t),t)$ remains uniformly bounded for all $t \geq 0$, and furthermore $\pa_x u(x(a,t),t) \to 0$ as $t \to +\infty$.
\end{itemize}
\end{corollary}

\section{Euler-Alignment system with attractive-repulsive
  potentials}\label{sec:potential}
In this section, we study the main system \eqref{main-eq} of Euler
equations with alignment, and attractive-repulsive forces. 

\subsection{A special potential: Euler-Poisson-Alignment system}\label{sec:EAP}
We first consider Euler-Alignment system \eqref{main-eq}
with a special Newtonian potential:
\[K(x)=\frac{k|x|}{2}.\]
As $\kk=k\delta_0$, $k(\pa_xK\star\rho)$ is a Poisson force which is attractive if
$k>0$, or repulsive if $k<0$. We recall the corresponding Poisson-Alignment system
\eqref{1d-main-newton}:
\begin{align*}
\begin{aligned}
&\pa_t \rho + \pa_x (\rho u) = 0,\cr
&\pa_t u + u \pa_x u = -k\pa_x \phi + \int_{\R} \psi(x-y)(u(y,t) - u(x,t))
\rho(y,t) dy, \quad  \pa_x^2\phi =\rho.
\end{aligned}
\end{align*}
By using similar arguments to Section \ref{sec:EA}, we find
\begin{subequations}\label{eq-propEAP}
\begin{eqnarray}
&& \rho^\prime = - \rho (d - \psi \star \rho),\label{eq-propEAPa}\\
&& d^\prime = - d(d - \psi \star \rho) - k\rho.\label{eq-propEAPb}
\end{eqnarray}
\end{subequations}
In the case of vacuum $\rho_0=0$, the dynamics of $d$
\eqref{eq-propEAPb} are the same as for the Euler-Alignment case
\eqref{eq-prop1b}. Thus Proposition \ref{prop1} holds.

We now focus on the case $\rho_0>0$. Set $\beta = d / \rho$, then we can find
\bq\label{re-newt}
\beta^\prime = -k, \quad \mbox{i.e.,} \quad \beta(t) = \beta_0 - kt.
\eq
This again yields that
\bq\label{newt-eq}
\rho^\prime = - \rho(d - \psi \star \rho) = -\rho(\rho(\beta_0 - kt) - \psi \star \rho) = -\beta_0\rho^2 + kt\rho^2 + \rho(\psi \star \rho).
\eq
Then we obtain an implicit form of solution $\rho$ from \eqref{newt-eq} that
\bq\label{sol-rho}
\frac{1}{\rho(t)} = e^{-\int_0^t (\psi \star \rho) ds}\lt( \frac{1}{\rho_0} + \int_0^t (\beta_0 - ks)e^{\int_0^s (\psi \star \rho) d\tau} ds\rt).
\eq

For the attractive case $k > 0$, $\beta_0 - ks$ becomes negative in finite
time, irrespective of the value of $\beta_0$. The right hand side
decreases to 0 in finite time, resulting a blowup of $\rho$. Hence, we
have the following lemma, and this directly implies the result in part (1) of Theorem \ref{thm:EAP}. 
\begin{lemma}{[Blowup with attractive potential]}.
If $k > 0$, then $\rho(t)\to+\infty$ in finite time.
\end{lemma}
For the repulsive case $k < 0$, critical thresholds are expected since a similar phenomenon is proved for both Euler-Poisson \cite{ELT} and
Euler-Alignment (Section \ref{sec:EA}) systems. We start with the
following rough estimate.
\begin{lemma}{[Rough subcritical region with repulsive potential].}\label{lem:roughsub}
If $k < 0$ and $\beta_0\geq0$, then $\rho(t)$ remains bounded for all time $t\geq0$. 
\end{lemma}
\begin{proof}
By our assumption on $k$ and $\beta_0$ we have $\beta_0-kt\geq0$
for all $t\geq0$. Also,
\bq\label{bound-est}
0 \leq \psi_m t\leq \int_0^t (\psi \star \rho) ds \leq \psi_M t.
\eq
Hence,  the following lower bound, $1/\rho(t)\geq e^{-\psi_Mt}/\rho_0>0$,
follows directly from \eqref{sol-rho}.
\end{proof}
We notice that lemma \ref{lem:roughsub} provides the same subcritical region as in the Euler-Alignment system, namely, $\pa_xu_0+\psi\star\rho_0>0$.
With the additional repulsive force, however, the subcritical region
is expected to be larger. Indeed, we turn to derive a refined estimate which yields the larger
subcritical region stated in Theorem \ref{thm:EAP}.\newline
Consider the case when $\beta_0<0$. To bound $\rho$,  it is enough check when the
following value is zero or not:
\bq\label{rho:est-1}
\frac{1}{\rho_0} + \int_0^t (\beta_0 - ks)e^{\int_0^s (\psi \star \rho) d\tau} ds.
\eq
Since $\beta_0 - ks \leq 0$ for $s \leq \frac{\beta_0}{k}$, we rewrite \eqref{rho:est-1} as 
\bq\label{rho:est-2}
\frac{1}{\rho_0} + \int_0^{\frac{\beta_0}{k}} (\beta_0 - ks)e^{\int_0^s (\psi \star \rho) d\tau} ds + \int_{\frac{\beta_0}{k}}^t (\beta_0 - ks)e^{\int_0^s (\psi \star \rho) d\tau} ds.
\eq
Thus if there is no blow-up of solutions until $t \leq \frac{\beta_0}{k}$ then it holds for all times, due to the positivity of the last term in \eqref{rho:est-2}. 

\begin{proposition}\label{prop:exact}
Consider the dynamics \eqref{newt-eq} with $k<0$ and $\beta_0<0$. Then
$\rho(\cdot,t)$ remains bounded if and only if
\bq\label{rho:iff}
\frac{1}{\rho_0} + \int_0^{\frac{\beta_0}{k}} (\beta_0 - ks)e^{\int_0^s (\psi \star \rho) d\tau} ds > 0.
\eq
\end{proposition}
In order to derive a sufficient condition for \eqref{rho:iff} determined by the initial conditions, we use \eqref{bound-est} to get
\begin{align*}
\begin{aligned}
\int_0^{\frac{\beta_0}{k}} (\beta_0 - ks)e^{\int_0^s (\psi \star \rho) d\tau} ds &\geq \int_0^{\frac{\beta_0}{k}} (\beta_0 - ks)e^{\psi_M s} ds \cr
&= -\frac{\beta_0}{\psi_M} + \frac{k}{\psi_M^2}\lt( e^{\psi_M\beta_0/k} - 1\rt)= -\frac{1}{\psi_M^2}\lt(k + \psi_M\beta_0 - ke^{\psi_M\beta_0/k}\rt).
\end{aligned}
\end{align*}
Thus, we deduce that if initially
\[
\frac{1}{\rho_0} -\frac{1}{\psi_M^2}\lt(k + \psi_M\beta_0 - ke^{\psi_M\beta_0/k}\rt) > 0,
\]
then there is no finite-time blow-up of classical solutions. Note that the left
hand side is increasing with respect to $\beta_0$ if $\beta_0<0$. This together with Lemma \ref{lem:roughsub} implies the subcritical region in part (2) of Theorem
\ref{thm:EAP}.

Next, we estimate the blow-up criterion of solutions. According to
Proposition \ref{prop:exact}, we shall find a sufficient condition of $d_0$ that makes
\[
\frac{1}{\rho_0} + \int_0^{\frac{\beta_0}{k}} (\beta_0 - ks)e^{\int_0^s (\psi \star \rho) d\tau} ds \leq 0.
\]
Since 
\[
\int_0^{\frac{\beta_0}{k}} (\beta_0 - ks)e^{\int_0^s (\psi \star \rho) d\tau} ds \leq \int_0^{\frac{\beta_0}{k}} (\beta_0 - ks) ds = \frac12 \frac{\beta_0^2}{k},
\]
we obtain $1/\rho_0 + \beta_0^2/2k \leq 0$, and this holds if $d_0
\leq - \sqrt{-2k\rho_0}$. It concludes that if $d_0 \leq -
\sqrt{-2k\rho_0}$ then there exists $t_*$ such that $\rho(t) \to
+\infty$ until $t \leq t_*$. 

\begin{remark}\label{rem:EAPsharp}
If $\psi$ has a positive lower bound, i.e., $\psi_m>0$, a better bound
can be obtained as follows:
\begin{align*}
\begin{aligned}
\int_0^{\frac{\beta_0}{k}} (\beta_0 - ks)e^{\int_0^s (\psi \star \rho) d\tau} ds &\leq \int_0^{\frac{\beta_0}{k}} (\beta_0 - ks)e^{\psi_m s} ds 
= -\frac{1}{\psi_m^2}\lt(k + \psi_m\beta_0 - ke^{\psi_m\beta_0/k}\rt).
\end{aligned}
\end{align*}
Therefore, we arrive at a refined supercritical region, where
$\sigma_-$ in Theorem \ref{thm:EAP} can be redefined as
\[
\frac{1}{\rho_0(x)} -\frac{1}{\psi_m^2}\lt(k + \psi_m \sigma_-(x)/\rho_0(x) - ke^{\psi_M \sigma_-(x)/k\rho_0(x)}\rt) = 0, \quad \mbox{for} \quad \rho_0(x)>0,
\]
and $\sigma_-(x)=0$ for $\rho_0(x)=0$.
In particular when $\psi$ is a constant, $\sigma_+=\sigma_-$, the two thresholds
matches and the results are sharp. 
\end{remark}

It follows from \eqref{re-newt} that if $\rho(t)$ blows up in finite
time, then $d(t)$ is also blowing up in finite-time. Similarly, if
$\rho(t)$ remains bounded, then $d(t)$ remains bounded as well.
As $|\psi\star\rho|\leq\psi_M$, $\rho(t)$ and $\pa_xu(t)$ blow up
simultaneously, concluding the proof of Theorem \ref{thm:EAP}.

\subsection{Euler-Alignment with general attractive-repulsive potentials}\label{sec:general}
In this part, we consider Euler-Alignment system with general
attractive-repulsive potentials:
\begin{align*}
\begin{aligned}
&\pa_t \rho + \pa_x (\rho u) = 0,\cr
&\pa_t u + u \pa_x u = \int_{\R} \psi(x-y)(u(y,t) - u(x,t)) \rho(y,t) dy + \pa_x K \star \rho.
\end{aligned}
\end{align*}
By using the similar arguments in Section \ref{sec:EA}, we find
\begin{align}\label{eq-prop2}
\begin{aligned}
& \rho^\prime = - \rho (d - \psi \star \rho),\cr
& d^\prime = - d(d - \psi \star \rho) + \kk \star \rho.
\end{aligned}
\end{align}

For this system, we can classify the initial configurations that
leading to the global regularity or the finite-time breakdown, when
$\kk$ is bounded. 
\begin{proposition}\label{prop2} Consider the system \eqref{eq-prop2}. Then the following holds.
\begin{itemize}
\item {[Attractive case $\kk > 0$]}. If $d_0 < 0$, then $d(t) \to - \infty$  in finite time.
\item {[Repulsive case $\kk < 0$]}. If $d_0 \geq 0$, then $d(t)$ remains uniformly bounded for all time $t \geq 0$.
On the other hand,
if $d_0<-\sqrt{\|\kk\|_{L^\infty}}$,
then $d(t) \to - \infty$ in finite time.
\end{itemize}
\end{proposition}
\begin{proof}
$\diamond$ We begin with the attractive case $(\kk \geq  0)$. In this case, we find from \eqref{eq-prop2} that
\begin{align}\label{eq-prop2-1}
\begin{aligned}
& \rho^\prime = - \rho (d - \psi \star \rho),\cr
& d^\prime \leq - d(d - \psi \star \rho).
\end{aligned}
\end{align}
Then one can use the comparison principle for the above system \eqref{eq-prop2-1} with the system \eqref{eq-prop1} to obtain
\[
d(t) \to - \infty \quad \mbox{in finite time} \quad \mbox{if} \quad d_0 < 0.
\]
We turn to the repulsive case, $(\kk < 0)$. We have
\[
d^\prime = -d^2 + (\psi \star \rho)d + \kk\star\rho.
\]
To obtain a global bound on $d$, we estimate when $d\geq0$,
\[
d^\prime\geq -d^2 + (\psi \star \rho)d.
\]
Thus $d(t)\geq0$ if $d_0\geq0$ due to the comparison principle with the system \eqref{eq-prop1}.
Moreover, we can also obtain the upper bound when $d>0$.
\[
d^\prime \leq -d^2 + \psi_Md + B=-(d-d_1^*)(d-d_2^*), \quad 
B:=\|\kk\|_{L^\infty},
\]
where
\[
d_1^* := \frac{\psi_M-\sqrt{\psi_M^2+4B}}{2}
 \quad \mbox{and} \quad d_2^* :=\frac{\psi_M+\sqrt{\psi_M^2+4B}}{2}>0.
\]
It implies $d^\prime<0$ for $d>d_2^*$, and this deduces $d$ has an upper bound. Hence we have the global boundedness of $d$.

On the other hand, for $d<0$, the upper bound is given as
\[
d^\prime \leq -d^2 + B = -
\left(d-\sqrt{B}\right) \left(d+\sqrt{B}\right),
\]
If $d_0<-\sqrt{B}$, through the comparison principle, 
it is easy to see that $d(t)\to-\infty$ in finite time.
\end{proof}

Collecting all characteristic flows together, we deduce Theorem \ref{thm13}.

We can also have more refined estimates by using the same argument as in Proposition \ref{prop2} when the influence function $\psi$ is bounded from below by $\psi_m > 0$. In this case, we can treat the case of combined attractive and repulsive forces.

\begin{corollary}\label{coroll1} Consider the system \eqref{main-eq} with the nonlocal interaction force $K \in \dot{W}^{2,\infty}(\R)$. Suppose that the influence function $\psi$ satisfies $\psi(x) \geq \psi_m >0$. If the initial slope $u'_0$ is not ``too negative" in the sense that
\[
\psi_m^2 \geq 4B \quad \mbox{and} \quad d_0 \geq - \frac{\psi_m + \sqrt{\psi_m^2 - 4B}}{2},
\]
then $d(t)$ is bounded for all time $t \geq 0$. On the other hand, if $u'_0$ is ``too negative" in the sense that 
\[
d_0 < \frac{\psi_m - \sqrt{\psi_M^2 + 4B}}{2},
\]
then $d(t) \to - \infty$ in a finite time. 
\end{corollary}

\begin{remark}\label{rem:genlower}
Corollary \ref{coroll1} implies that if the influence function $\psi$ is
bounded from below and it is sufficiently large, then we have better
threshold conditions:
\begin{itemize}
\item {[Subcritical region]}. If $\psi_m^2 \geq 4B$ and
\[\pa_x u_0(x) \geq -\psi \star \rho_0(x) - \frac{1}{2}\lt(\psi_m +
\sqrt{\psi_m^2 - 4B}\rt)\]
for all $x\in\R$, then the system has a global classical solution.
\item {[Supercritical region]}. If there exists an $x$ such that
\[\pa_x u_0(x) < -\psi \star \rho_0(x) + \frac{1}{2}\lt(\psi_m - \sqrt{\psi_M^2 +
  4B}\rt),\]
then the solution blows up in finite time.
\end{itemize}
\end{remark}

\section{A remark on Euler-Alignment system with pressure}\label{sec:pressure}
In this section, we consider Euler-Alignment system with
pressure,
\begin{align}\label{press}
\begin{aligned}
&\pa_t \rho + \pa_x(\rho u) = 0, \quad x \in \R, \quad t >0,\cr
&\pa_t u + u\pa_x u + \frac{\pa_x p(\rho)}{\rho} = \int_{\R}\psi(x-y)(u(y,t) - u(x,t))\rho(y,t)\,dy.
\end{aligned}
\end{align}

The pressure is usually modeled through a power law
$p(\rho)=A\rho^\gamma$, with $A\geq0$ and $\gamma\geq1$.
In particular, if there is no alignment interaction, namely
$\psi\equiv0$, the system becomes 1D compressible Euler equation with
isentropic pressure, where general initial data leads to finite-time
blowup. On the other hand, if $A=0$, we recover the pressure-less system
\eqref{eq:EA}.

It is of great interest to see whether alignment force regularizes
Euler equation with pressure, and if there exists a non-empty
subcritical region of initial configurations such
that the solution of \eqref{press} is globally smooth.
As the pressure term destroys the original characteristic structure,
it is more delicate the trace the dynamics.
To this end, we follow the argument of \cite{TW}, where
Euler-Poisson equations with pressure is discussed. 

Rewrite \eqref{press} as the following system
\bq\label{press-sys}
\begin{pmatrix}\rho\\u\end{pmatrix}_t+
\begin{pmatrix}u&\rho\\A\gamma\rho^{\gamma-2}&u\end{pmatrix}
\begin{pmatrix}\rho\\u\end{pmatrix}_x=
\begin{pmatrix}0\\ \int_{\R}\psi(x-y)(u(y,t) -
  u(x,t))\rho(y,t)\,dy\end{pmatrix}.
\eq
We decouple the system by diagonalizing the matrix in
\eqref{press-sys}. It yields
\[\begin{aligned}
R_t+\lambda R_x= \int_{\R}\psi(x-y)(u(y,t) - u(x,t))\rho(y,t)\,dy,\\
S_t+\mu S_x=\int_{\R}\psi(x-y)(u(y,t) -u(x,t))\rho(y,t)\,dy,
\end{aligned}
\]
where $\lambda,\mu$ are eigenvalues of the matrix
\[\lambda:=u-\sqrt{A\gamma}\rho^{(\gamma-1)/2},\quad
\mu:=u+\sqrt{A\gamma}\rho^{(\gamma-1)/2},
\]
and $R,S$ are the corresponding Riemann invariants
\[R=\begin{cases}
u-\frac{2\sqrt{A\gamma}}{\gamma-1}\rho^{(\gamma-1)/2}&\gamma>1\\
u-\sqrt{A}\ln\rho&\gamma=1
\end{cases},\quad
S=\begin{cases}
u+\frac{2\sqrt{A\gamma}}{\gamma-1}\rho^{(\gamma-1)/2}&\gamma>1\\
u+\sqrt{A}\ln\rho&\gamma=1
\end{cases}.\] 

Let us denote material derivatives $^\backprime=\pa_t+\lambda\pa_x$ and
$^\prime=\pa_t+\mu\pa_x$ along two particle paths.
We derive the dynamics of $R_x$ and $S_x$ using the same procedure as in \cite{TW}:
\bq\label{press-pair-1}
\begin{aligned}
R_x^\backprime+\frac{1+\theta}{2}R_x^2+\frac{1-\theta}{2}R_xS_x=\pa_x \int_{\R}\psi(x-y)(u(y,t) - u(x,t))\rho(y,t)\,dy,\\
S_x^\prime+\frac{1+\theta}{2}S_x^2+\frac{1-\theta}{2}R_xS_x=\pa_x\int_{\R}\psi(x-y)(u(y,t) -u(x,t))\rho(y,t)\,dy,
\end{aligned}
\eq
where $\theta=\frac{\gamma-1}{2}$.

We treat the right hand side of the system similarly as the
pressure-less system. Define
\[r=R_x+\psi\star\rho\quad\text{and}\quad s=S_x+\psi\star\rho.\]
The paired equations \eqref{press-pair-1} can be written in terms of
$(r, s)$.
\begin{align*}
\begin{aligned}
r^\backprime+\frac{1+\theta}{2}r^2+\frac{1-\theta}{2}rs=
\left[\left(1+\frac{\theta}{2}\right)r-\frac{\theta}{2}s\right](\psi\star\rho)
-\sqrt{A\gamma}\rho^\theta(\psi_x\star\rho),\\
s^\prime+\frac{1+\theta}{2}s^2+\frac{1-\theta}{2}rs=
\left[\left(1+\frac{\theta}{2}\right)s-\frac{\theta}{2}r\right](\psi\star\rho)
+\sqrt{A\gamma}\rho^\theta(\psi_x\star\rho).
\end{aligned}
\end{align*}

We begin our study on a very special isothermal case, where $p(\rho)=A\rho$.
Moreover, we take $\psi\equiv C>0$. Under this setup, the
alignment force is reduced to a local damping force, and system
\eqref{press} is simplified to 
\bq\label{eq:Eulerdamping}
\begin{aligned}
&\pa_t \rho + \pa_x(\rho u) = 0,\cr
&\pa_t u + u\pa_x u + A\pa_x (\ln\rho) = -Cu + C\int_\R \rho_0 u_0\,dx.
\end{aligned}
\eq
Here we used the momentum conservation of the system \eqref{press}, i.e.,
\[
\int_\R \rho u\,dx = \int_\R \rho_0 u_0\,dx, \quad \mbox{for} \quad t \geq 0.
\]
Under the above assumptions, the dynamics of $(r,s)$ can be simplified as follows.
\begin{subequations}\label{press-isothermal}
\begin{eqnarray}
&&r^\backprime=-\frac{1}{2}r^2-\frac{1}{2}rs+
Cr,\label{press-isothermala}\\
&&s^\prime=-\frac{1}{2}s^2-\frac{1}{2}rs+Cs
\label{press-isothermalb}.
\end{eqnarray}
\end{subequations}

The following proposition shows a subcritical threshold where global
regularity is guaranteed.
\begin{proposition}[Subcritical region]
If $r_0$ and $s_0$ are nonnegative and $r_0,s_0 \in L^\infty(\R)$, then $r(x,t),s(x,t)$ are bounded for all $x\in\R,t\geq0$. More precisely, we have
\[
0 \leq r(x,t), s(x,t) \leq \max\{ \|r_0\|_{L^\infty}, \|s_0\|_{L^\infty}, 2C\} \quad \mbox{for }\, (x,t) \in \R \times \R_+.
\] 
\end{proposition}

\begin{proof}
As $r$ and $s$ has the same dynamics along their own characteristics, it suffices to prove that $r$ is bounded along its characteristic path.

For the proof, we look for an invariant region of the form $[0,M_0]$ with $M_0 > 0$ which will be determined later. If $r,s \in [0, M_0]$, then we find from \eqref{press-isothermala} that
\[
r^\backprime \geq -\frac12 r^2 - \frac12M_0 r + Cr.
\]
We will show that $r$ can not escape from the left. Let $q$ solve the following differential equation:
\[
\frac{d q}{dt} = -\frac12 q^2 - \frac12M_0 q + Cq,
\]
with the initial data $q_0 = 0$. Then by the uniqueness of solutions to the above equation, we obtain $q(t) \equiv 0$. This and together with the comparison lemma yield $r(t) \geq 0$ for all time. 

We now show that $r$ has an upper bound. We notice that
\[
r^\backprime = -\frac12 r(r - 2C) - \frac12 sr \leq -\frac12 r(r - 2C).
\]
Thus if $r_0 \in (0,2C]$, then $r(t) \leq 2C$ for all time. On the other hand, if $r_0 > 2C$, then we clearly get $r^\backprime \leq 0$ and this implies that $r(t) \leq r_0$ for all time. Hence we obtain that $r$ is uniformly bounded in time with 
\[
0 \leq r(t) \leq \max\{ \|r_0\|_{L^\infty}, 2C\}.
\]
The boundedness of $s$ can be done by the same process, and this implies
\[
0 \leq r(t), s(t) \leq \max\{ \|r_0\|_{L^\infty}, \|s_0\|_{L^\infty}, 2C\}.
\] 
We conclude our desired result by choosing $M_0 = \max\{ \|r_0\|_{L^\infty}, \|s_0\|_{L^\infty}, 2C\}$.
\end{proof}

Recall $r_0=\pa_xu_0-\sqrt{A}(\pa_x\rho_0)/\rho_0+C$ and
$s_0=\pa_xu_0+\sqrt{A}(\pa_x\rho_0)/\rho_0+C$. The subcritical region is equivalent to
\bq\label{eq:EDsub}
\pa_x u_0(x)\geq-C+\sqrt{A}\left|\frac{\pa_x\rho_0(x)}{\rho_0(x)}\right|,\quad\forall~x\in\mathbb{R}.
\eq
For $C>0$, the set $(\rho_0,u_0)$ is non-empty. In the case of $A=0$,
we recover the subcritical threshold condition for pressureless Euler
dynamics with damping.

On the other hand, the boundedness of $r$ and $s$ imply the boundedness of
$\pa_xu$ and $\pa_x\ln\rho$, we obtain the following global existence result.

\begin{theorem}\label{thm:isothermal}
Consider \eqref{eq:Eulerdamping} with subcritical initial condition
\eqref{eq:EDsub}. Then, the system has a global classical solution.
\end{theorem}

\begin{remark}
Critical thresholds on Euler equation with local damping has been
discussed in \cite{LL09}, for general types of pressure. Our result
provides a larger subcritical region for the special case
\eqref{press-isothermal}.
\end{remark}

For the general non-local influence function $\psi$, the strategy does
not apply, due to the presence of the extra term
$\pm\sqrt{A\gamma}\rho^\theta(\psi_x\star\rho)$. In this case, $r=0$
and $s=0$ are not stationary states, and it is more delicate to
control the lower bound of $(r, s)$ in time. We leave the problem for
future investigation.

\begin{appendix}
\section{Global regularity}\label{sec:regularity}
In this part, we consider smoother subcritical initial data, and
prove that initial regularity persists globally in time, under suitable
assumptions on the influence function $\psi$ and the interaction potential $K$.

\subsection{Pressureless Euler equations with nonlocal forces}\label{sec:App-1}
We start with our main system \eqref{main-eq}.
\begin{theorem}\label{thm:regularity}
Let $s\geq0$ be an integer. Consider system \eqref{main-eq} with
smooth influence function $\psi$ satisfying
\bq\label{eq:phicond}
\psi\in L^1(\R)+\text{const},\quad\text{and}\quad x\pa_x\psi\in L^1(\R),
\eq
and potential $K$ such that 
\bq\label{eq:kcond}
\kk\in L^1(\R).
\eq
Suppose the initial data $(\rho_0,u_0)$ lie in the subcritical region, and
satisfy $\rho_0, \pa_xu_0\in H^s(\R)$.
Then, there exists a
 unique global solution $(\rho,u)$ such that
\[\rho \in \mc([0,T];H^s(\R)) \quad\text{and}\quad\pa_xu\in\mc([0,T];H^s(\R))\]
for any time $T$.
\end{theorem}

\begin{remark}
The condition \eqref{eq:phicond} on $\psi$ is valid for constant influence function
$\psi\equiv1$, as well as the typical Cucker-Smale weight
$\psi(x)=(1+x^2)^{-\gamma}$, with $\gamma>1/2$. 
The condition on $K$ is valid for Newtonian potential $K=\frac{k}{2}|x|$.
\end{remark}

As subcritical initial data imply global in time bounds on
$\|\rho\|_{L^\infty}$ and $\|\pa_xu\|_{L^\infty}$, it suffices to
prove the following estimate.

\begin{theorem}\label{thm:globalbound}
Let $s\geq0$ be an integer. Define
$Y(t):=\|\pa_x\u(\cdot,t)\|_{H^s}^2+\|\rho(\cdot,t)\|_{H^s}^2$. If
the influence function $\psi$ and the potential $K$ satisfy \eqref{eq:phicond} 
and \eqref{eq:kcond}, respectively, then
\begin{equation}\label{eq:Yprop}
Y(T)\lesssim Y(0)
\exp\left[\int_0^T\left(1+\|\rho(\cdot,t)\|_{L^\infty}+\|\grad\u(\cdot,t)\|_{L^\infty}
\right)dt\right].
\end{equation}
\end{theorem}

\begin{proof}
We start with acting operator
$\Lambda^s$ on equation \eqref{main-eq}$_1$ and integrate by parts against
$\Lambda^s\rho$. Here $\Lambda:=(I-\Delta)^{1/2}$ is the
pseudo-differential operator. We also denote $(\cdot,\cdot)$ as $L^2$
inner product in $\R$.

The evolution of the $H^s$ norm reads
\[\frac{1}{2}\frac{d}{dt}\|\rho(\cdot,t)\|^2_{H^s} = 
-\left(\left[\Lambda^s\div,\u\right]\rho,\Lambda^s\rho\right)
+\frac{1}{2}\left(\Lambda^s\rho,(\div\u)\Lambda^s\rho\right).\]
We postpone the proof of the following commutator estimate to
Lemma \ref{lem:1}.
\[\|\left[\Lambda^s\div,\u\right]\rho\|_{L^2}\lesssim\|\grad\u\|_{L^\infty}\|\rho\|_{H^s}
+\|\pa_xu\|_{H^s}\|\rho\|_{L^{\infty}}.\]
With this, we deduce the following estimate
\[
\frac{d}{dt}\|\rho(\cdot,t)\|^2_{H^s} \lesssim
\left[\|\rho\|_{L^\infty}+\|\grad\u\|_{L^\infty}\right]
\left(\|\rho\|_{H^s}^2+\|\pa_xu\|_{H^s}^2\right).
\]

Similarly, for equation \eqref{main-eq:b}, we have
\begin{align*}
\frac{1}{2}\frac{d}{dt}\|\pa_x\u(\cdot,t)\|^2_{H^s}& = 
-\big(\left[\Lambda^s\pa_x,u\right]\pa_x\u,\Lambda^s\pa_x\u\big)
+\frac{1}{2}\left(\Lambda^s\pa_x\u,(\div\u)\Lambda^s\pa_x\u\right)\\
& +\left(\Lambda^s\pa_x\u, \Lambda^s\pa_x
\int_{\R}\psi(\x-\y)(\u(\y)-\u(\x))\rho(\y)d\y\right)\\
& +\left(\Lambda^s\pa_x\u, \Lambda^s\pa_x
\int_{\R}\grad K(\x-\y)\rho(\y)d\y\right).
\end{align*}

For the commutator, we obtain the same estimate from Lemma
\ref{lem:1}.
\[\|\left[\Lambda^s\pa_x,u\right]\pa_x\u\|_{L^2}\lesssim
\|\pa_x\u\|_{L^{\infty}}\|\pa_x\u\|_{H^s}.\]

For the alignment term, we claim and will prove in Lemma
\ref{lem:2} that
\bq\label{eq:alignprop}
\begin{split}
&\left(\Lambda^s\pa_x\u, \Lambda^s\pa_x
\int_{\R}\psi(\x-\y)(\u(\y)-\u(\x))\rho(\y)d\y\right)\\
&~\quad\lesssim
\|\pa_x\u\|_{H^s}\big[\|\rho\|_{H^s}\|\grad\u\|_{L^\infty}+
(1+\|\rho\|_{L^\infty})\|\pa_x\u\|_{H^s}\big].
\end{split}
\eq

For the attraction-repulsion term, as $\kk\in L^1(\R)$,
\[
\left(\Lambda^s\pa_x\u, \Lambda^s\pa_x
\int_{\R}\grad K(\x-\y)\rho(\y)d\y\right)\lesssim
\|\pa_x\u\|_{H^s}\|\pa_x^2 K\|_{L^1}\|\Lambda^s\rho\|_{L^2}
\lesssim\|\pa_x\u\|_{H^s}\|\rho\|_{H^s}.
\]

Putting everything together, we obtain
\[
\frac{d}{dt}\|\pa_x\u(\cdot,t)\|_{H^s}^2
\lesssim\left[1+\|\rho\|_{L^\infty}+\|\grad\u\|_{L^\infty}\right]
\left(\|\rho\|_{H^s}^2+\|\pa_x\u\|_{H^s}^2\right).
\]

A Gronwall's inequality implies \eqref{eq:Yprop}.
\end{proof}

Next, we provide a short proof of the Kato-Pounce type commutator
estimate \cite{KP} which is used in the regularity estimates.
\begin{lemma}[Commutator estimate]\label{lem:1} 
  Let $f, \pa_x u \in (L^\infty\cap H^s)(\R)$. Take $s$ to be an non-negative integer. Then,
\[\|[\pa_x^{s+1},u]f\|_{L^2}\lesssim\|f\|_{L^\infty}\|\pa_xu\|_{H^s}+\|\pa_xu\|_{L^\infty}\|f\|_{H^s}.\]
\end{lemma}
\begin{remark}
Take $f=\rho$, we get 
\[\|[\pa_x^{s+1},u]\rho\|_{L^2}\lesssim\|\rho\|_{L^\infty}\|\pa_xu\|_{H^s}+\|\pa_xu\|_{L^\infty}\|\rho\|_{H^s}.\]
Take $f=\pa_xu$, we get
\[\|[\pa_x^{s+1},u]\pa_xu\|_{L^2}\lesssim\|\pa_xu\|_{L^\infty}\|\pa_xu\|_{H^s}.\]
These imply the two commutator inequalities in  Theorem \ref{thm:globalbound}.
\end{remark}

\begin{proof}[Proof of Lemma \ref{lem:1}]
We first rewrite the commutator and use appropriate H\"older
inequality to get
\begin{align*}
\|[\pa_x^{s+1},u]f\|_{L^2}&=\left\|\pa_x^{s+1}(uf)-u\cdot\pa_x^{s+1}f\right\|_{L^2}
\leq\sum_{i=0}^s{s+1\choose i}\|\pa_x^{s+1-i}u\cdot\pa_x^if\|_{L^2}\\
&\leq \sum_{i=0}^s{s+1\choose i}\|\pa_x^{s-i}(\pa_xu)\|_{L^{\frac{2s}{s-i}}}\|\pa_x^if\|_{L^{\frac{2s}{i}}}.
\end{align*}
Next, we apply the following type of Gagliardo-Nirenberg interpolation inequality
\[\|\pa_x^jg\|_{L^{\frac{2s}{j}}}\lesssim\|\pa_x^sg\|_{L^2}^{j/s}\|g\|_{L^\infty}^{1-\frac{j}{s}},\qquad
j=0,1,\cdots,s.\]
Taking $(g,j)=(\pa_xu,s-i)$ and $(g,j)=(f,i)$, we continue the estimate
\begin{align*}
\|[\pa_x^{s+1},u]f\|_{L^2}
&\lesssim\sum_{i=0}^s{s+1\choose i}
(\|\pa_x^s(\pa_xu)\|_{L^2}\|f\|_{L^\infty})^{\frac{s-i}{s}}
(\|\pa_x^sf\|_{L^2}\|\pa_xu\|_{L^\infty})^{\frac{i}{s}}\\
&\lesssim
\|\pa_x^s(\pa_xu)\|_{L^2}\|f\|_{L^\infty}+
\|\pa_x^sf\|_{L^2}\|\pa_xu\|_{L^\infty}.
\end{align*}

The last estimate is due to Young's inequality. This ends the proof.
\end{proof}

We are left with the final estimate \eqref{eq:alignprop}.

\begin{lemma}\label{lem:2}
If the influence function $\psi$ satisfies \eqref{eq:phicond}, then
\[\left\|\pa_x^{s+1}\int_\R\psi(x-y)(u(y)-u(x))\rho(y)dy\right\|_{L^2}\lesssim
(\|\rho\|_{L^\infty} + 1)\|\pa_xu\|_{H^s}+\|\pa_xu\|_{L^\infty}\|\rho\|_{H^s}.\]
\end{lemma}
\begin{proof}
We first assume $\psi\in L^1(\R)$. Estimate the left hand side as follows.
\[
\text{LHS} \leq
\sum_{i=0}^s{s+1\choose
  i}\|\pa_x^{s+1-i}u\cdot\pa_x^i(\psi\star\rho)\|_{L^2}
+\left\|\int_\R\pa_x^{s+1}\psi(x-y)(u(y)-u(x))\rho(y)dy\right\|_{L^2}
=\textrm{I}+\textrm{II}.
\]
We control the first term in the similar way as Lemma
\ref{lem:1}, along with the assumption that $\psi\in L^1(\R)$.
\begin{align*}
\textrm{I}&\leq \sum_{i=0}^s{s+1\choose
  i}\|\pa_x^{s-i}(\pa_xu)\|_{L^{\frac{2s}{s-i}}}\|\pa_x^i\rho\|_{L^{\frac{2s}{i}}}\|\psi\|_{L^1}
\lesssim \|\pa_x^s(\pa_xu)\|_{L^2}\|f\|_{L^\infty}+
\|\pa_x^sf\|_{L^2}\|\pa_xu\|_{L^\infty}.
\end{align*} 

For the second term, we again break it into two terms (again, suppressing the $t$-dependence),
\begin{align*}
\textrm{II}&=\left\|\int_\R\pa_y^{s+1}\psi(x-y)(u(y)-u(x))\rho(y)dy\right\|_{L^2}
=\left\|\int_\R\psi(x-y)\pa_y^{s+1}\left[(u(y)-u(x))\rho(y)\right]dy\right\|_{L^2}\\
&\leq\left\|\int_\R\psi(x-y) \sum_{i=0}^s{s+1\choose
  i}\pa_y^{s+1-i}u(y)\pa_y^i\rho(y)dy\right\|_{L^2}
+\left\|\int_\R\psi(x-y)(u(y)-u(x))\pa_y^{s+1}\rho(y)dy\right\|_{L^2}\\
&=\textrm{III}+\textrm{IV}.
\end{align*} 

The third term can be controlled by Lemma \ref{lem:1} after applying
Young's inequality
\[\textrm{III}\leq\|\psi\|_{L^1}\sum_{i=0}^s{s+1\choose
  i}\|\pa_x^{s+1-i}u\cdot\pa_x^i\rho\|_{L^2}
\lesssim \|\pa_x^s(\pa_xu)\|_{L^2}\|\rho\|_{L^\infty}+
\|\pa_x^s\rho\|_{L^2}\|\pa_xu\|_{L^\infty}.\]

Finally, for the last term, we have
\begin{align*}
\textrm{IV}&=\left\|\int_\R\pa_y\left[\psi(x-y)(u(y)-u(x))\right]\pa_y^s\rho(y,t)dy\right\|_{L^2}\\
&\leq
\left\|\int_\R\pa_y\psi(x-y)(u(y)-u(x))\pa_y^s\rho(y)dy\right\|_{L^2}
+\left\|\int_\R\psi(x-y)\pa_yu(y)\pa_y^s\rho(y)dy\right\|_{L^2}\\
&\leq
\|x\pa_x\psi\|_{L^1}\|\pa_xu\|_{L^\infty}\|\pa_x^s\rho\|_{L^2}
+\|\psi\|_{L^1}\|\pa_xu\|_{L^\infty}\|\pa_x^s\rho\|_{L^2}.
\end{align*} 

For $\psi\in L^1+const$, it is easy to check that for constant $c$,
\[\pa_x^{s+1}\int_\R c(u(y,t)-u(x,t))\rho(y,t)dy=-c\pa_x^{s+1}u(x).\]
Thus, we conclude with the desired estimate.
\end{proof}


\subsection{Isothermal Euler equations with nonlocal dissipation}\label{sec:App-2}
In this part, we study the global regularity for the system
\eqref{press} with the pressure law $p(\rho) = \rho$. The system can be
rewritten in the following form.
\begin{align}\label{re-IE}
\begin{aligned}
&\pa_t n + u \pa_x n = - \pa_x u, \quad n := \ln \rho,\cr
&\pa_t u + u \pa_x u + \pa_x n = \int_{\R} \psi(x-y)(u(y,t) - u(x,t))
\rho(y,t) dy.
\end{aligned}
\end{align}

\begin{theorem}\label{thm:p} Let $s \geq 0$ be an integer. Consider
  the system \eqref{re-IE} with the influence function $\psi$ satisfying
\bq\label{eq:psicondpressure}
\pa_x \psi \in W^{s-1,\infty}(\R), \quad |x|^{1/2} \pa_x^\alpha \psi \in L^2(\R) \quad \mbox{ for } 1 \leq \alpha \leq s+1.
\eq
Suppose the initial data $(n_0:=\ln \rho_0,u_0)$ satisfy $\pa_x n_0 \in H^s(\R)$, and $\pa_x u_0 \in H^s(\R)$. If the solutions $(n,u)$ have the following global in time bounds
\[
\sup_{0 \leq t \leq T}\lt(\|\pa_xn\|_{L^\infty} + \|\pa_xu\|_{L^\infty}\rt) < \infty,
\]
Then there exists a unique global solution $(n,u)$ such that
\[
\pa_x n \in \mc([0,T];H^s(\R)) \quad \mbox{and} \quad \pa_x u \in \mc([0,T];H^s(\R)),
\]
for any time $T$.
\end{theorem}

For the similar reason as before, it is enough to prove the following estimate for the global solvability. 

\begin{remark} If the influence function $\psi$ is a constant, then we do not need the assumption \eqref{eq:psicondpressure} for $\psi$. Since the subscritical initial data obtained in Theorem \ref{thm:isothermal} for the system \eqref{eq:Eulerdamping} imply global in time bounds on $\|\pa_x (n,u)\|_{L^\infty}$, Theorem \ref{thm:p} deduce that the initial regularity persists globally in time. 
\end{remark}

\begin{theorem}\label{prop_2} Let $s\geq0$ be an integer. Define
$Y(t):=\|\pa_xn(\cdot,t)\|_{H^s}^2+\|\pa_xu(\cdot,t)\|_{H^s}^2$. If
the influence function $\psi$ satisfies \eqref{eq:psicondpressure}, then
\[
Y(T)\lesssim Y(0) \exp\lt(\int_0^T \lt(1 + \|\pa_x
n(\cdot,t)\|_{L^\infty}+\|\pa_x u(\cdot,t)\|_{L^\infty} \rt) dt \rt).
\]
\end{theorem}
\begin{proof}
The proof is similar as in Theorem \ref{thm:globalbound}.

We start with acting operator
$\Lambda^s\pa_x$ on equation \eqref{re-IE}$_1$ and integrate by parts against
$\Lambda^s\pa_xn$:
\[\frac{1}{2}\frac{d}{dt}\|\pa_xn(\cdot,t)\|^2_{H^s} = 
-\left(\left[\Lambda^s\pa_x,u\right]\pa_xn,\Lambda^s\pa_xn\right)
+\frac{1}{2}\left(\Lambda^s\pa_xn,(\div\u)\Lambda^s\pa_xn\right)
-\left(\Lambda^s\pa_sn, \Lambda^s\pa_x^2u\right).\]
The commutator estimate in Lemma \ref{lem:1} implies
\[\|\left[\Lambda^s\div,\u\right]\pa_xn\|_{L^2}\lesssim\|\grad\u\|_{L^\infty}\|\pa_xn\|_{H^s}
+\|\pa_xu\|_{H^s}\|\pa_xn\|_{L^{\infty}}.\]
With this, we deduce the following estimate
\bq\label{eq:nest}
\frac{1}{2}\frac{d}{dt}\|\pa_xn(\cdot,t)\|^2_{H^s} +\left(\Lambda^s\pa_sn,
  \Lambda^s\pa_x^2u\right) \lesssim
\left[\|\pa_xn\|_{L^\infty}+\|\grad\u\|_{L^\infty}\right]
\left(\|\pa_xn\|_{H^s}^2+\|\pa_xu\|_{H^s}^2\right).
\eq
Similarly, for equation \eqref{re-IE}$_2$, we have
\begin{align*}
\frac{1}{2}\frac{d}{dt}\|\pa_x\u(\cdot,t)\|^2_{H^s}& = 
-\big(\left[\Lambda^s\pa_x,u\right]\pa_x\u,\Lambda^s\pa_x\u\big)
+\frac{1}{2}\left(\Lambda^s\pa_x\u,(\div\u)\Lambda^s\pa_x\u\right) 
-\left(\Lambda^s\pa_xu,\pa_x\Lambda^s\pa_x^2n\right)\\
& +\left(\Lambda^s\pa_x\u, \Lambda^s\pa_x
\int_{\R}\psi(\x-\y)(\u(\y)-\u(\x))\rho(\y)d\y\right).
\end{align*}

With the same argument as in Theorem \ref{thm:globalbound}, we get
\bq\label{eq:uest}
\frac{1}{2}\frac{d}{dt}\|\pa_x\u(\cdot,t)\|_{H^s}^2
+\left(\Lambda^s\pa_xu,\pa_x\Lambda^s\pa_x^2n\right)
\lesssim\left[1+\|\rho\|_{L^\infty}+\|\grad\u\|_{L^\infty}\right]
\left(\|\pa_xn\|_{H^s}^2+\|\pa_x\u\|_{H^s}^2\right),
\eq
provided the following estimate for the alignment term is true.
\bq\label{eq:lem3}
\left(\Lambda^s\pa_x\u, \Lambda^s\pa_x
\int_{\R}\psi(\x-\y)(\u(\y)-\u(\x))\rho(\y)d\y\right)
\lesssim\|\pa_x\u\|_{H^s}^2.
\eq
See Lemma \ref{lem:3} for the proof of this inequality.

Finally, we add \eqref{eq:nest}-\eqref{eq:uest} and use
Gronwall's inequality to end the proof.
\end{proof}

We end this section by proving the estimate \eqref{eq:lem3}. A stronger
regularity on $\psi$ is required (compared with Lemma \ref{lem:2}) as
we do not have regularity for $\rho$ any more.
\begin{lemma}\label{lem:3}
If $\psi$ satisfies \eqref{eq:psicondpressure}, then the estimate \eqref{eq:lem3} is satisfied.
\end{lemma}
\begin{proof}
It suffices to prove for all $1\leq\alpha\leq s+1$,
\[\left(\pa_x^\alpha\u, \pa_x^\alpha
\int_{\R}\psi(\x-\y)(\u(\y)-\u(\x))\rho(\y)d\y\right)
\lesssim\|\pa_x\u\|_{H^s}^2.\]
In fact, we get
\begin{align*}
\text{LHS}=&\left(\pa_x^\alpha\u, \sum_{i=0}^\alpha{\alpha\choose i}
\int_{\R}\pa_x^i\psi(\x-\y)\pa_x^{\alpha-i}(\u(\y)-\u(\x))\rho(\y)d\y\right)\\
\leq&-(\psi\star\rho)\|\pa_x^\alpha u\|_{L^2}^2
+\|\pa_x^\alpha u\|_{L^2}\sum_{i=1}^{\alpha-1}{\alpha\choose i}
\left\|\pa_x^{\alpha-i}u\cdot\pa_x^i(\psi\star\rho)\right\|_{L^2}\\
&+\|\pa_x^\alpha
u\|_{L^2}\left\|\int_\R\pa_x^\alpha\psi(x-y)(u(y,t)-u(x,t))\rho(y,t)dy\right\|_{L^2}\\
\leq& ~0+\mathrm{I}+\mathrm{II}.
\end{align*}
Next, we estimate \textrm{I} and \textrm{II} term by term. 
\begin{align*}
\textrm{I}~\leq~&\|\pa_x^\alpha u\|_{L^2}\sum_{i=1}^{\alpha-1}{\alpha\choose i}
\|\pa_x^{\alpha-i}u\|_{L^2}\|\pa_x^i\psi\|_{L^\infty}\|\rho\|_{L^1}\lesssim
\|\pa_x\psi\|_{W^{s-1,\infty}}\|\pa_xu\|_{H^s}^2\lesssim\|\pa_x\u\|_{H^s}^2,\\
\textrm{II}~\leq~&\|\pa_x^\alpha u\|_{L^2}
[u]_{C^{1/2}}\big\||x|^{1/2}\pa_x^\alpha\psi(x)\star\rho\big\|_{L^2}
\cr
&\hspace{-0.5cm} \lesssim \|\pa_x^\alpha u\|_{L^2}
\|\pa_xu\|_{L^2}\| |x|^{1/2}\pa_x^\alpha\psi(x)\|_{L^2}\|\rho\|_{L^1}\lesssim\|\pa_x\u\|_{H^s}^2.
\end{align*}
This concludes the proof.
\end{proof}

\end{appendix}

\noindent
\textbf{Acknowledgments.} Research was supported by NSF grant RNMS11-07444 (KI-Net). Additional support was provided  by the project MTM2011-27739-C04-02
DGI (Spain) and 2009-SGR-345 from AGAUR-Generalitat de Catalunya and by the Royal Society through a Wolfson Research Merit Award (JAC), by Engineering and Physical Sciences Research Council grant EP/K008404/1 (JAC and YPC), by  NSF grant DMS10-08397 (ET) and by  ONR grant N00014-1512094 ONR (ET and CT).

%
%
%
%

\end{document}